\documentclass{amsart}
\usepackage{amsthm,amsmath,amssymb, color, graphics, todonotes}
\usepackage{tikz-cd}
\usepackage{mathrsfs}
\usepackage{adjustbox}
\usepackage{hyperref}
 \usepackage{amsaddr}
 \usepackage{enumitem}

\newtheorem{thm}{Theorem}[section]
\newtheorem{claim}[thm]{Claim}

\newtheorem{lem}[thm]{Lemma}

\newtheorem{question}[thm]{Question}
\theoremstyle{definition}
\newtheorem{defn}[thm]{Definition}

\theoremstyle{remark}

\newcommand\ZZ{\mathbb{Z}}

\newcommand\Fin{\mathrm{Fin}}

\newcommand{\Pbb}{\mathbb{P}}
\newcommand{\Qbb}{\mathbb{Q}}

\newcommand{\Zbb}{\mathbb{Z}}

\newcommand\Acal{\mathcal{A}}

\newcommand\Abf{\mathbf{A}}

\newcommand\Hbb{\mathbb{H}}

\newcommand\dom{\operatorname{dom}}

\title{Merging $\lim^1 \mathbf{A} \ne 0$ with other nonvanishing constructions}
\author{Nathaniel Bannister}
\author{Justin Tatch Moore}

\thanks{
The authors would like to thank
Stevo Todorcevic for offering useful comments on an early draft of this paper.
The research of the second author is supported in part by NSF grant DMS-2451350.}

\begin{document}

\begin{abstract}
We develop methods for forcing $\lim^1 \Abf \ne 0$,
where $\Abf$ is a particular inverse system of 
abelian groups introduced by
Marde\v{s}i\'c and Prasolov in their computation of certain
strong homology groups. 
These methods allow us to extend
previous nonvanishing results of
Casarosa and Lambie-Hanson for $\lim^k \Abf$ for
$k \geq 2$.
Specifically we show that,
for a given $n$, it is relatively consistent with
ZFC that 
$\mathfrak{b} = \mathfrak{d} = \omega_n$ and 
$\lim^k \Abf \ne 0$ whenever $1 \leq k \leq n$ (previously established with $2 \leq k \leq n$).
We also show it is relatively consistent with ZFC that $\mathfrak{b} = \mathfrak{d} = \omega_{\omega+2}$ and
$\lim^k \Abf  \ne 0$ for all $k \geq 1$ (previously
established with $k \geq 2$).
We also adapt proofs of Kamo to show that
$\lim^1 \Abf = 0$ holds in many
finite support iterated forcing
extensions.
\end{abstract}

\maketitle

\section{Introduction} 
This paper contributes to a growing body of research on the interactions of set theory and homological algebra, and in particular to the derived functors of the inverse limit functor. 
We isolate several situations in which the first derived limit of a particular system $\mathbf{A}$ does and does not vanish. 
The derived limits of this system play a role in several fields of mathematical research; to give two examples:
\begin{itemize}
    \item In \cite{SHINA}, Marde\v{s}i\'c and Prasolov isolate the system $\mathbf{A}$ and show that if \emph{strong homology} is additive on the class of closed subspaces of $\mathbb{R}^n$, then $\lim^n\mathbf{A}=0$ for all $n>0$;
    \item Clausen and Scholze have shown that $\lim^n\mathbf{A}[H]=0$ for all $n>0$ is equivalent to a statement about derived functors in condensed abelian groups (see, for instance, \cite[Lecture 4]{AnStacks}).
\end{itemize}
The focus of the early research in the late 1980s and 1990s was on the \emph{first} derived limit of $\mathbf{A}$; see, for instance, \cite{SHINA, DSV, Kamo, To98}. 
The past decade has seen renewed interest in the system $\mathbf{A}$, with a particular focus on the derived limits $\lim^n\mathbf{A}$ for $n\geq 2$, including \cite{B17, SVHDL, strong_hom_add, SVHDLwLC, NVHDL, SNVHDL, ADLCM}. 
A number of questions have remained open about the derived limits of $\mathbf{A}$. 
One of these is the following, which is the main motivation for this paper: 
\begin{question} \label{SNV-question}
    Let $X$ be a set of positive integers. 
    Is there a model of ZFC in which $\lim^n\mathbf{A} \ne 0$ if and only if $n\in X$?
\end{question}
In \cite{SNVHDL}, Casarosa and Lambie-Hanson prove several special cases of Question \ref{SNV-question}, notably the cases of intervals $X=[2,n]$ for any $n\geq 2$ and $X=[2,\infty)$. 
The first sections of this paper are devoted to constructing several forcings after which $\lim^1\mathbf{A}\neq0$ and which are compatible with their constructions. 
We then turn to establishing a robust class of forcings for making $\lim^1\mathbf{A}=0$, extending work of Kamo \cite{Kamo}. 

We now briefly discuss some of the methods and the structure of the paper. 
In section \ref{nonlin-section}, we define a nonlinear Hechler-style iteration to show the following result:
\begin{thm} \label{nonlinear-thm}
    Let $D$ be a well-founded partial ordering such that any countable subset of $D$ has a strict upper bound. 
    There is a ccc forcing extension in which $\lim^1\mathbf{A}\neq0$ and $(\omega^\omega,<^*)$ has a cofinal suborder isomorphic to $D$. 
\end{thm}
When $D$ is a wellorder, this can be
viewed as a variation on Laver's
method for producing a model
in which the continuum is large and
$(\omega^\omega,<^*)$ contains
a saturated linear order of cardinality
$2^{\aleph_0}$ \cite{Laver}.
We remark here that the well-foundedness of $D$ is likely unnecessary using machinery similar to that of Hechler's original paper \cite{Hechler}. 
In section \ref{snv_section}, we outline how to combine the forcing we construct in section \ref{nonlin-section} with the weak diamond and square hypotheses used in \cite{SNVHDL} to obtain the following:
\begin{thm}
    Relative to the consistency of ZFC, it is consistent that $\mathfrak{b}=\mathfrak{d}=\omega_{n}$ and $\bigwedge_{1\leq k\leq n}\lim^k\mathbf{A}\neq0$. 
    Relative to the consistency of ZFC, it is consistent that $\mathfrak{b}=\mathfrak{d}=\omega_{\omega+2}$ and $\bigwedge_{1\leq k<\omega}\lim^k\mathbf{A}\neq0$. 
\end{thm}
Section \ref{trivializing-freezing-section} relates 1-coherent families to gaps in $\mathscr{P}(\omega)/\mathrm{fin}$ and
then records lemmas about trivializing 1-coherent families and making them indestructible by forcing.
This follows from Kunen's unpublished analysis of the analogous phenomena for gaps (see \cite{Baumgartner:PFA}).
We then use these lemmas in section \ref{lin-section} to construct a linear iteration for making $\mathfrak{d}$ arbitrarily large while also forcing $\lim^1\mathbf{A}\neq0$.

The final two sections are devoted to adapting proofs due to Kamo \cite{Kamo} to show that in many finite support
iterated forcing extensions, $\lim^1\mathbf{A}=0$.  
In section \ref{knaster-section}, use a result on freezing 1-coherent families from section \ref{trivializing-freezing-section} to show the following: 
\begin{thm}
    Suppose $\operatorname{cof}(\kappa)>\aleph_1$ and $\langle\Pbb_i,\dot\Qbb_i\mid i<\kappa\rangle$ is a length $\kappa$ finite support iteration of nontrivial Knaster forcings. 
    Then $\Pbb_\kappa$ forces $\lim^1\mathbf{A}=0$. 
\end{thm}

\section{A nonlinear iteration} \label{nonlin-section}
Let $D$ be a well-founded partial ordering with $\mathfrak{b}(D)\geq\aleph_1$. We let $D^+$ be $D$ with a new maximum element $+$. 
For $a\in D$, define $D/a$ as $\{b\in D\mid b<a\}$. 
\begin{defn} \label{forcing_def}
    For $a\in D^+$, define $\mathbb{P}_a$ recursively as follows. 
    A condition $p$ consists of
    \begin{itemize}
        \item A finite $D_p\subseteq D/a$;
        \item A finite $F_p\subseteq \omega\times\omega$; 
        \item For each $b\in D_p$,
        \begin{itemize}
            \item an $s_{b,p}\in \omega^{<\omega}$; 
            \item a $\Phi_{b,p}\colon F_p\to\mathbb{Z}$;
            \item a $\mathbb{P}_b$-name $\dot f_{b,p}$ for an element of $\omega^\omega$.
        \end{itemize}
    \end{itemize}
    We set $p\leq q$ if and only if
    \begin{enumerate} \label{ordering_requirements}
        \item $D_p\supseteq D_q$;
            \item $F_p\supseteq F_q$;
        \item For every $b\in D_q$:
        \begin{enumerate}
            \item $s_{b,p}\sqsupseteq s_{b,q}$;
            \item $\Phi_{b,p}\sqsupseteq\Phi_{b,q}$;
            \item $p\upharpoonright D/b$ forces in 
            $\mathbb{P}_b$ that
            \begin{itemize}
                \item $\dot f_{b,p}\geq\dot f_{b,q}$;
                \item $\forall n\in \dom(s_{b,p})\setminus\dom(s_{b,q})(s_{b,p}(n)\geq \dot f_{b,q}(n))$
            \end{itemize}
            
        \end{enumerate}
        \item For all $b,c\in D_q$ and $(n,m)\in F_p\setminus F_q$ such that both
            \begin{itemize}
                \item $n\not\in \dom(s_{p,c})$ or $m\leq s_{p,c}(n)$ and 
                \item $n\not\in \dom(s_{p,b})$ or $m\leq s_{p,b}(n)$,
            \end{itemize}
            we have
            \[\Phi_{b,p}(n,m)=\Phi_{c,p}(n,m).\]
    \end{enumerate}
\end{defn}
We will show that $\mathbb{P}_+$ is a ccc forcing which adds a cofinal suborder of $(\omega^\omega,<^*)$ isomorphic to $D$ and a nontrivial coherent family indexed by that cofinal suborder. 
In particular, $\mathbb{P}_+$ forces $\lim^1\mathbf{A}\neq0$.
\begin{lem} \label{ccc}
    For any $a\in D^+$, $\mathbb{P}_+$ is ccc.
\end{lem}
\begin{proof}
    Let $\langle p_\alpha\mid \alpha<\omega_1\rangle$ be conditions. 
    By thinning the sequence out if necessary, we may assume 
    \begin{itemize}
        \item $\langle D_{p_\alpha}\mid\alpha<\omega_1\rangle$ forms a $\Delta$-system with root $R$;
        \item $F_{p_\alpha}$ is constant;
        \item for $b\in R$, the values $s_{b,p_\alpha},\Phi_{b,p_\alpha},$ and $\Phi_{b,p_\alpha}$ are constant.
    \end{itemize}
    Then any two remaining conditions are compatible.
\end{proof}

For $b\in D$, we write $x_b=\bigcup_{p\in G}s_{b,p}$ and note that $x_b$ is forced to be a function from $\omega$ to $\omega$. 
\begin{lem}
    For $b,c\in D$, $x_b<^*x_c$ if and only if $b<c$. 
\end{lem}
\begin{proof}
    For the ``if'' direction, observe that there is a $\mathbb{P}_c$-name for $x_b$; in particular, for any condition $p$ we may extend $p$ to a condition $q$ by setting $\dot f_{c,q}$ to be a $\mathbb{P}_c$-name for $x_b\vee\dot f_{c,p}$; then $q$ forces $x_b<^*x_c$. 
    The converse is immediate if $c<b$ so assume that $b$ and $c$ are incomparable. 
    Fix a condition $p$ and $n<\omega$ sufficiently large that $n>|s_{b,p}|$ for all $b\in D_p$; we show there is a $q\leq p$ such that $q$ forces $x_{b}(n)>x_c(n)$. An easy density argument will then yield that $x_b(n)\not<^*x_c(n)$. 

    To this end, find $q\in \mathbb{P}_b$ with $q\leq p\upharpoonright (D/b)$ such that $q$ decides $\dot f_{b,p}\upharpoonright (n+1)$. 
    Let $p_1$ be a condition refining $p$ with $D_{p_1}=D_p\cup D_q$, $p_1(d)=q(d)$ if $d<b$ and $s_{d,p_1}=s_{d_p}$, $\dot f_{d,p_1}=\dot f_{d,p}$ otherwise; note that some extension of various $\Phi_{a,p}$ may be required.
    Now find $r\leq p_1\upharpoonright (D/c)$ such that $r$ decides $\dot f_{c,p}\upharpoonright (n+1)$ and let $p_2$ be a condition refining $p_1$ such that $D_{p_2}=D_{p_1}\cup D_r$, $p_2(d)=r(d)$ if $d<c$ and $s_{d,p_2}=s_{d,p_1}$, $\dot f_{d,p_2}=\dot f_{d,p_1}$ otherwise. 
    Now it is easy to extend $p_2$ to a condition which forces $x_{b}(n)>x_c(n)$.
\end{proof}
\begin{lem}
    $\{x_b\mid b\in D\}$ is cofinal in $(\omega^\omega,<^*)^{V_{\mathbb{P}_+}}$.
\end{lem}
\begin{proof}
    Suppose $\dot y$ is a $\mathbb{P}_+$-name for an element of $\omega^\omega$. 
    Using that $\mathbb{P}_+$ is ccc, we fix for each $n$ a maximal antichain $\langle p_m^n\mid m<\omega\rangle$ deciding the value of $\dot y(n)$. 
    Since $\mathfrak{b}(D)\geq\aleph_1$, we may find $a\in D$ a strict upper bound for $\{D(p^n_m)\mid m,n<\omega\}$. 
    Then $\dot y$ is equivalent to a $\mathbb{P}_a$-name for an element of $\omega^\omega$ so $\Vdash\dot y<^*x_a$.
\end{proof}
\begin{lem}
    $\Phi_{x_b}=(\bigcup_{p\in G}\Phi_{b,p})\upharpoonright I(x_b)$ defines a nontrivial coherent family indexed by $\{x_b\mid b\in D\}$. 
\end{lem}
\begin{proof}
    Coherence follows from an easy density argument since for any condition $p$ and any $b,c\in D_p$, $p\Vdash\Phi_{x_b}=^{F_p}\Phi_{x_c}$ from the definition of the ordering. 
    For nontriviality, suppose that $\Vdash\dot\Psi\colon\omega\times\omega\to\mathbb{Z}$. 
    For each $m,n$, let $\mathcal{A}_{m,n}$ be a maximal antichain of conditions that decide the value of $\dot\Psi(m,n)$. 
    Since $\mathcal{A}_{m,n}$ is countable by Lemma \ref{ccc} and $\mathfrak{b}(D)\geq\aleph_1$, we may fix $a\in D$ a strict upper bound for $\bigcup_{m,n}\mathcal{A}_{m,n}$. 
    We show that $\Phi_a$ is forced to disagree infinitely often with $\Psi$; in particular, $\Psi$ does not trivialize $\Phi$.
    
    Fix a condition $p$ and $n<\omega$. 
    We find $q\leq p$ and $m,k<\omega$ such that $m\geq n$, $s_{q,a}(m)\geq k$, and $q\Vdash\Phi_a(m,k)\neq \Psi(m,k)$. 
    By increasing $n$ if necessary, we may assume that 
    \begin{itemize}
        \item $n>|s_{p,a}|$;
        \item there is no $\ell$ with $(n,\ell)\in F_p$.
    \end{itemize}
    Fix $r\leq p\upharpoonright (D/a)$ deciding $\dot f_{p,a}\upharpoonright (n+1)$ to be some $t$. 
    By extending $r$ if necessary, we may assume that if $b\in D_p\cap (D/a)$ then $|s_{b,q}|>n$. 
    Let 
    \[k=1+\max_{b\in D_p\cap (D/a)}s_{b,r}(n).\]
    By extending $r$ if necessary, we may also assume that $r$ decides $\Psi(n,k)$ and that $(n,k)\in F_r$. 
    We now define a condition $q\leq p$ with $D_q=D_p\cup D_r$, $F_q=F_r$, and
    \begin{itemize}
        \item If $b\in D_r$ then $q(b)=r(b)$;
        \item If $b\in D_q$ and $b\not\leq a$ then $s_{b,p}=s_{b,q}$ and $\dot f_{b,q}=\dot f_{b,q}$;
        \item $s_{a,q}$ extends $s_{a,p}$, $|s_{a,q}|= n+1$, $s_{a,q}(n)\geq k$, and for each $\ell\in\operatorname{dom}(s_{a,q})\setminus\operatorname{dom}(s_{a,p})$, $s_{a,q}(\ell)\geq t(\ell)$;
        \item $\dot f_{a,q}=\dot f_{a,p}$;
        \item $\Phi_{a,q}(n,k)$ disagrees with the value $r$ decides for $\Psi(n,k)$. 
    \end{itemize}
    Then $q$ forces $\Phi_{a,q}(n,k)\neq\Psi(n,k)$, as desired. 
\end{proof}
\section{Simultaneous nonvanishing} \label{snv_section}
The main results of this section are the following two theorems:
\begin{thm} \label{snv_omega_n_thm}
    Relative to the consistency of ZFC, it is consistent that $\mathfrak{b}=\mathfrak{d}=\omega_{n}$ and $\bigwedge_{1\leq k\leq n}\lim^k\mathbf{A}\neq0$. 
\end{thm}
\begin{thm} \label{snv_thm}
    Relative to the consistency of ZFC, it is consistent that $\mathfrak{b}=\mathfrak{d}=\omega_{\omega+2}$ and $\bigwedge_{1\leq k<\omega}\lim^k\mathbf{A}\neq0$. 
\end{thm}
We first prove Theorem \ref{snv_omega_n_thm}:
\begin{proof}
    Our model will be the following forcing extension of $L$:
    \[V'=L^{\mathbb{P}_{\omega_n}\times \mathbb{C}_1\times\ldots\times\mathbb{C}_n},\]
    where $\mathbb{P}_{\omega_n}$ is the forcing of Definition \ref{forcing_def} with $D$ the linear order $\omega_n$ and $\mathbb{C}_k=\operatorname{Fn}(\omega_{n+k},2,\omega_k)$ for all $1\leq k\leq n$; i.e., $\mathbb{C}_k$ is the forcing to add $\omega_{n+k}$ many subsets of $\omega_k$. 
    By Easton's lemma, every real added by this product is added by $\mathbb{P}_{\omega_n}$ so in $V'$, $\mathfrak{b}=\mathfrak{d}=\omega_n$ and the $1$-coherent family added by $\mathbb{P}_{\omega_n}$ is nontrivial. 
    Nonvanishing $\lim^k\mathbf{A}$ for $2\leq k\leq n$ follows from \cite[Lemma 5.7]{SNVHDL} as the verification of the appropriate combination of $w\diamondsuit$ and $\square$ sequences is identical to the proof of \cite[Lemma 5.6]{SNVHDL}.
\end{proof}
Theorem \ref{snv_thm} follows similarly:
\begin{proof}
    Our model will be the following forcing extension of $L$:
    \[V'=\mathbb{P}_{\omega_\omega}\times\prod_{1\leq n<\omega}\mathbb{C}_n\times\mathbb{C}_\omega,\]
    where $\mathbb{P}_{\omega_n}$ is the forcing of Definition \ref{forcing_def} with $D$ the linear order $\omega_{\omega+2}$ and $\mathbb{C}_k=\operatorname{Fn}(\omega_{\omega+n+2},2,\omega_n)$ for all $1\leq n<\omega$, and $\mathbb{C}_\omega=\operatorname{Fn}(\omega_{\omega\cdot 2+2},2, \omega_{\omega+2})$. 
    Easton's lemma yields that $\mathfrak{b}=\mathfrak{d}=\omega_{\omega+2}$ and $\lim^1\mathbf{A}\neq0$ and the nonvanishing of $\lim^k\mathbf{A}$ for $k\geq 2$ follows from \cite[Lemma 5.10]{SNVHDL}.
\end{proof}

\section{Trivializing and freezing 1-coherent families} \label{trivializing-freezing-section}

We now recall a transformation of Todorcevic (see \cite[\S4]{DSV}) which converts nontrivial coherent families
indexed by subsets of $\omega^\omega$ into gaps in $[\omega]^\omega/\Fin$.
Suppose that $X \subseteq \omega^\omega$ is finitely directed under $<^*$ 
and $\Phi = \langle \phi_x \mid x \in X \rangle$ is a 1-coherent family
such that $\phi_x \colon I_x \to \Zbb$.
Define $A_x \subseteq I_x \times \Zbb$ to be the graph of $\phi_x$:
$$A_x := \{(i,j,k) \in \omega \times \omega \times \Zbb \mid (i,j) \in I_x \textrm{ and } k = \phi_x(i,j)\}.$$
Define $B_x:= I_x \times \Zbb \setminus A_x$.
It is easily checked that for all $x,y \in X$, $A_x \cap B_y$ is finite and 
if $x <^* y$ are in $X$, then $A_x \subseteq^* A_y$ and $B_x \subseteq^* B_y$.
Furthermore $\Phi$ is trivial if and only if there is a $C \subseteq \omega \times \omega \times \Zbb$ such
that for all $x \in X$, $A_x \subseteq^* C$ and $B_x \cap C$ is finite.

Now suppose $X\subseteq\omega^\omega$ and
$\Phi$ is a $1$-coherent family indexed by $X$. 

\begin{defn}
    $\mathbb{T}_\Phi$ consists of conditions $p=(\psi_p,F_p)$ where
    \begin{itemize}
        \item $\psi_p$ is a finite partial function from $\omega\times\omega$ to $\mathbb{Z}$;
        \item $F_p$ is a finite subset of $X$ such that for every $x,y\in F_p$, \[\{(n,m)\mid \Phi_x(n,m)\neq\Phi_y(n,m)\}\subseteq\dom(\psi_p).\]
    \end{itemize}
    We define $q\leq p$ to mean
    \begin{itemize}
        \item $\psi_q\sqsupseteq\psi_p$;
        \item $F_q\supseteq F_p$;
    \item for all $x\in F_p$ and all $(n,m)\in I_x\cap \dom(\psi_q) \setminus \dom(\psi_p)$, $\psi_q(n,m)=\Phi_x(n,m)$. 
    \end{itemize}
\end{defn}
This poset is the coherent family analog of the Kunen's gap-splitting
poet $P(\bar a,\bar b)$ defined in \cite[\S4]{Baumgartner:PFA}.
Kunen's characterization of destructible gaps (see \cite[4.2]{Baumgartner:PFA}) yields the following lemmas after minor modifications to the arguments in \cite{Baumgartner:PFA}:

\begin{lem} \label{destructible_equiv_lemma}
    Suppose $X$ is finitely directed under $<^*$. The following are equivalent:
    \begin{enumerate}[label=(\arabic*)]
        \item $\mathbb{T}_\Phi$ is ccc; \label{ccc_item}
        \item $\Phi$ can be trivialized by a forcing which preserves $\aleph_1$; \label{destructible}
        \item For all uncountable $Y\subseteq X$ there are $x,y\in Y$ such that $\Phi_x=\Phi_y$ on the intersection of their domains. \label{eq}
    \end{enumerate}
\end{lem}

\begin{lem}
\label{indestructible_lem}
    Suppose $X=\{x_\alpha\mid\alpha<\omega_1\}$ is a subset of $\omega^\omega$ and $\Phi$ is a coherent family indexed by $X$ such that whenever $\mathcal{A}$ is an uncountable family of finite pairwise disjoint subsets of $\omega_1$, setting $\Phi_F=\Phi_{x_{\min F}}\upharpoonright I\left(\bigwedge_{\alpha\in F}x_\alpha\right)$, the family
    \[\Phi^\Acal=\left\{\Phi_F\,\middle|\,F\in\mathcal{A}\right\}\]
    is nontrivial.
    There is a ccc forcing extension in which $\Phi$ cannot be trivialized by any $\aleph_1$-preserving forcing. 
    In particular, $\Phi$ cannot be trivialized by any Knaster forcing.
\end{lem}

\section{A linear iteration} \label{lin-section}
In \cite[2.1]{To98}, Todorcevic proved PFA implies
there is a forcing extension with the same
$\mathscr{P}(\omega_1)$ in which there is a 
nontrivial coherent family indexed by a cofinal suborder of $(\omega^\omega,<^*)$. 
In particular, MA holds in this forcing extension.
Moreover the conclusion in Todorcevic's result (including that MA holds) follows from a fragment of PFA
which does not require large cardinals for its consistency.

In this section demonstrate a different method for forcing
a fragment of $\mathrm{MA}_{\aleph_1}$ together with $\lim^1\mathbf{A} \ne 0$.
This method is analogous
to work of Baumgartner, Frankiewicz, and Zbierski \cite{BFZ}
in which they show how to 
force via a c.c.c. poset MA($\sigma$-linked)
together with the assertion
that any Boolean algebra of cardinality at most $2^{\aleph_0}$ 
can be embedded into $\mathscr{P}(\omega)/\Fin$.

\begin{thm} \label{MA_thm}
    Suppose $\kappa$ is uncountable and $\kappa^{<\kappa}=\kappa$. 
    There is a ccc forcing extension in which $2^{\aleph_0}=\kappa$, MA($\sigma$-linked) holds, and $\lim^1\mathbf{A}\neq0$. 
\end{thm}

Theorem \ref{MA_thm} follows from the usual MA bookkeeping (see \cite[Ch. VIII]{set_theory:kunen}) plus the following lemma:
\begin{lem}
    Suppose 
    $\langle \mathbb{P}_\alpha,\dot{\mathbb{Q}}_\alpha \mid \alpha<\varepsilon\rangle$ is a finite support iteration such that for some $E\subseteq\varepsilon$,
    \begin{itemize}
        \item If $\alpha\in\varepsilon\setminus E$ then $\dot\Qbb_\alpha$ is $\sigma$-linked;
        \item If $\alpha\in E$ then $\dot \Qbb_\alpha$ is the product $\mathbb{T}_{\Phi\upharpoonright\alpha}\times  \mathbb{\dot H}$.
    \end{itemize}
    Here $\Phi\upharpoonright\alpha$ is defined recursively to be the coherent family indexed by the Hechler reals $\{h_\beta\mid\beta\in E\cap \alpha\}$ and whose value at $h_\beta$ is the restriction of the trivialization to $\Phi\upharpoonright\beta$ generically added at stage $\beta$ to the area under the graph of $h_\beta$.
    
    Then $\Pbb_\varepsilon$ forces that $\mathbb{T}_{\Phi\upharpoonright\varepsilon}$ is ccc. 
    If additionally $\operatorname{cof}(\varepsilon)>\aleph_0$ and $E$ is unbounded in $\varepsilon$, then $\Phi\upharpoonright\varepsilon$ is nontrivial and indexed by a cofinal suborder of $(\omega^\omega,<^*)$. 
\end{lem}
\begin{proof}
    For each $\alpha\in\varepsilon\setminus E$,
    fix $\mathbb{P}_\alpha$-names
    $\langle \dot A^\alpha_n\mid n<\omega\rangle$
    for linked subsets of ${\dot {\mathbb{Q}}}_\alpha$
    which cover $\dot {\mathbb{Q}}_\alpha$.
    
    Let $D$ denote the set of all conditions $p\in\Pbb_\varepsilon$ satisfying
    the following conditions: 
       \begin{enumerate}
        \item \label{decide_linked}
        For each $\alpha\in\dom(p)\setminus E$, there is an $n_{p,\alpha}$ such that $p\upharpoonright\alpha\Vdash p(\alpha)\in \dot A^\alpha_{n_{p,\alpha}}$;
        \item \label{decide_working_parts}
        For each $\alpha\in\dom(p)\cap E$, there are $\psi_{\alpha,p},F_{\alpha,p},s_{\alpha,p}$ such that $p\upharpoonright\alpha$ forces that the $\mathbb{T}_{\Phi\upharpoonright\alpha}$ component of $p(\alpha)$ is $(\check\psi_{\alpha,p},\{h_\beta\mid\beta\in\check F\})$ and the Hechler component has stem $s_{\alpha,p}$;
        \item \label{include_Fap}
        If $\beta\in F_{\alpha,p}$ then $\beta\in\dom(p)$;
        \item \label{dom(psi)_indep} 
        The domain of $\psi_{\alpha,p}$ is independent of $\alpha$. 
    \end{enumerate}
    \begin{claim}
    $D$ is dense.
    \end{claim}
    \begin{proof}
    Define $D_0$ to consist of those conditions which satisfy (\ref{decide_linked}) and 
    (\ref{decide_working_parts}) and $D_1$ to consist of
    those conditions which satisfy (\ref{decide_linked})--(\ref{include_Fap}).
    It follows from a routine induction on $\varepsilon$ that $D_0$ is dense.
    Next observe that any condition in $D_0$ can be extended to a condition in $D_1$ by increasing the domain
    of $p$ but making the extension trivial on the new coordinates.

    To show that $D$ is dense, we will prove by induction on $\varepsilon$
    that for any finite $A \subseteq \omega \times \omega$,
    those $p \in D_1$ with $\alpha \mapsto \dom(\psi_{\alpha,p})$ taking a constant value containing $A$ are dense.
    To this end, let $p$ be a given condition and $A \subseteq \omega \times \omega$ be a given finite set.
    By extending $p$ if necessary, we may assume $p \in D_1$.
    Set $\beta:=\max(\dom(p))$.
    By our induction hypothesis there is a $q \in D_1 \cap \mathbb{P}_\beta$ with $q \leq_{\mathbb{P}_\beta} p\upharpoonright\beta$ and such that all the $\psi_{\alpha,q}$ have the same constant domain $B$ which contains $\dom(\psi_{\beta,p}) \cup A$.
    Now extend $p$ to a condition $r$ with $r\upharpoonright\beta=q$ and $\dom(\psi_{\beta,r})=B$ and
    $\psi_{\beta,r}(n,m)=\psi_{\beta,p}(n,m)$ if defined, else $\psi_{\beta,r}(n,m)=\psi_{\alpha,q}(n,m)$
    if there is an $\alpha \in F_{\beta,p}$ such that $m \not > s_{\alpha,q}(n)$, and $0$ otherwise.
    The second item is well-defined since if $\alpha,\gamma \in F_{\beta,p}$, then $p$ forces $\Phi_\alpha$ and
    $\Phi_\gamma$ coincide outside $\dom(\psi_{\beta,p})$.
    It is also a condition since by (\ref{include_Fap}), $F_{\beta,p} \subseteq \dom(q)$
    and the coherent family we are trivializing is precisely the $\psi_\alpha$'s restricted to the Hechler reals.
    \end{proof}   
    
    Going forward, we will assume all conditions come from $D$ and will verify item \ref{eq} of Lemma \ref{destructible_equiv_lemma}. 
    Suppose $p$ forces that $\dot Y\subseteq E$ is uncountable. 
    Fix an uncountable $A\subseteq E$ such that for each $\alpha\in A$ there is $p_\alpha\leq p$ such that $p_\alpha\Vdash\alpha\in\dot Y$; we may assume $\alpha\in\dom(p_\alpha)$. 
    By thinning out the family if necessary, we may assume
    \begin{itemize}
        \item $\psi_{p_\alpha,\alpha}$ is independent of $\alpha$; 
        \item $\langle\dom(p_\alpha)\mid\alpha\in A\rangle$ is a $\Delta$-system with root $r$;
        \item for each $\beta\in r\cap E$, the values of $\psi_{p_\alpha,\beta}$ and $s_{p_\alpha,\beta}$ are independent of $\alpha$;
        \item for each $\beta\in r\setminus E$, the value of $n_{p,\alpha}$ is independent of $\alpha$. 
    \end{itemize}
    Fix any $\alpha<\beta$ in $A$; it is sufficient to construct some $q\leq p_\alpha,p_\beta$ such that $q$ forces $\Phi_{h_\alpha}=\Phi_{h_\beta}$. 
    We set $\dom(q)=\dom(p_\beta)\cup\dom(p_\alpha)$ and define $q(\gamma)$ recursively as follows:
    \begin{itemize}
        \item If $\gamma\in \dom(q)\cap E$, we set $F_{\gamma,q}=\dom(q)\cap\gamma\cap E$ and $s_{q,\gamma}=s_{p_{\alpha},\gamma}\cup s_{p_\beta,\gamma}$; verifying that $q\upharpoonright\gamma$ forces this defines a condition in $\mathbb{T}_{\Phi\upharpoonright\gamma}$ is an easy induction on $\gamma$.
        \item If $\gamma\in\dom(q)\setminus (E\cup r)$, then exactly one of $p_\alpha(\gamma)$ or $p_\beta(\gamma)$ is defined; we set $q(\gamma)$ to be whichever one is defined. 
        \item If $\gamma\in(\dom(q)\cap r)\setminus E$, then since $q\upharpoonright\gamma\leq p_\alpha\upharpoonright\gamma,p_\beta\upharpoonright\gamma$, $q$ forces that $p_\alpha(\gamma),p_\beta(\gamma)$ lie in the same linked set; we set $q(\gamma)$ to be a $\Pbb_\gamma$ name for a common lower bound. 
    \end{itemize}
    Then $q$ forces $\Phi_{h_\alpha}=\Phi_{h_\beta}$ since $\psi_{\alpha,q}=\psi_{\beta,q}$ and $\alpha\in F_{\beta,q}$. 

    For the ``moreover'' statement, we first verify that $\Phi\upharpoonright\varepsilon$ is nontrivial. Induction using Lemma \ref{destructible_equiv_lemma} implies that $\Pbb_\varepsilon$ is ccc, so any putative trivialization $\Theta\colon \omega\times\omega\to\ZZ$ is in $V[G_\alpha]$ for some $\alpha<\varepsilon$. 
    Fix $\beta\in E\setminus (\alpha+1)$. 
    Then for each $n<\omega$,
    \[\{(\psi,F,s)\in \mathbb{T}_{\Phi\upharpoonright\beta}\times\Hbb\mid\exists(m,k)\left(m>n\wedge s(m)>k \wedge \psi(m,k)\neq\Theta(m,k)\right)\}\]
    is dense, so by genericity, $\Phi_{h_\beta}$ disagrees infinitely often with $\Theta$. 
    Similarly, since $h_\beta$ is Hechler generic over $V[G_\beta]$, every $x\in V[G_\beta]$ is eventually below $h_\beta$ so $\{h_\beta\mid\beta\in E\}$ is cofinal in $(\omega^\omega,<^*)^{V[G]}$.
\end{proof}

\section{Iterations of Knaster forcings} \label{knaster-section}
The main result of this section is the following:
\begin{thm} \label{Knaster_thm}
    Suppose $\operatorname{cof}(\kappa)>\aleph_1$ and $\langle\Pbb_i,\dot\Qbb_i\mid i<\kappa\rangle$ is a length $\kappa$ finite support iteration of nontrivial Knaster forcings. 
    Then $\Pbb_\kappa$ forces $\lim^1\mathbf{A}=0$. 
\end{thm}
Most of the ideas in the proof of Theorem \ref{Knaster_thm} are already present in \cite{Kamo}, with an additional idea of using that a finite support iteration of length $\kappa$ codes a generic for adding $\kappa$ many Cohen reals through the supports.  
We first show that in a finite support iteration of nontrivial Knaster forcings, coherent families as in Lemma \ref{indestructible_lem} not only remain nontrivial but actually cannot extend to all of the continuum. 

\begin{lem} \label{no_ext_lem}
    Suppose $X=\{x_\alpha\mid\alpha<\omega_1\}$ and $\Phi$ are as in the hypotheses of Lemma \ref{indestructible_lem} and $\langle\Pbb_i,\dot\Qbb_i\mid i<\omega\rangle$ is a finite support iteration of Knaster forcings such that each $\dot\Qbb_i$ has an infinite antichain; let $\langle \dot q^i_j\mid j<\omega\rangle$ be $\Pbb_i$-names for an infinite maximal antichain in $\dot\Qbb_i$. 
    In $V^{\Pbb_\omega}$, let $y\in\omega^\omega$ be the sequence such that $y(n)$ is the unique $m$ such that $q^n_m\in G_n$ and let $\Phi^y$ be the coherent family indexed by $X$ defined by 
    \[\Phi^y_{x_\alpha}(n,m)=\left\{\begin{array}{cc}
         \Phi_{x_\alpha}(n,m)&m\leq y(n)  \\
         0&\text{otherwise} 
    \end{array}\right..\]
    Then $\Phi^y$ satisfies the hypothesis of Lemma \ref{indestructible_lem} in $V^{\mathbb{P}_\omega}$. 
    In particular, in all further extensions by a Knaster forcing, $\Phi$ does not extend to a coherent family indexed by $X\cup\{y\}$. 
\end{lem}
\begin{proof}
    Suppose not and fix a condition $p$ such that $p$ forces some $\Acal=\langle \dot F_\alpha\mid i<\omega_1\rangle$ and $\Psi$ witnesses a failure for $\Phi^y$ to satisfy the hypotheses of Lemma \ref{indestructible_lem}. 
    For each $i<\omega_1$, fix some condition $p_i\leq p$ such that
    \begin{itemize}
        \item $p_i$ decides the value of $F_i$ to be some $F_i^*$;
        \item for some $n_i$, 
        \[p_i\Vdash\Phi_{F_i}=^{n_i}\Psi.\]
    \end{itemize}
    By the pidgeonhole principle, we may assume that for some $n<\omega$ and finite $R\subseteq\omega$, $n_i=n$ and $\dom(p_i)=R$ for each $i$. 
    Since a finite support iteration of Knaster forcings is Knaster, we may assume further that for each $i,j$, $p_i$ and $p_j$ are compatible. 
    Let $\Acal^*=\{F^*_i\mid i<\omega_1\}$; we claim that $\Acal^*$ witnesses a failure for $\Phi$ to satisfy the hypotheses of Lemma \ref{indestructible_lem}. 
    Indeed, for each $i,j$, since $p_i$ and $p_j$ have a common extension with domain bounded by $\max(R)+1$, we must have 
    \[\Phi_{F_i^*}=^{\max(n,\max(R)+1)}\Phi_{F_j^*},\]
    so gluing along indices with horizontal coordinate at least $\max(n,\max(R)+1)$ produces a trivialization of $\Phi^{\Acal^*}$. 
\end{proof}
The next lemma shows that if a coherent family is trivial on a sufficiently large index set, then it is already trivial. 
\begin{defn}
    $Y\subseteq\omega^\omega$ is \emph{everywhere unbounded} if for each infinite $E\subseteq\omega$ and $x\colon E\to\omega$, there is a $y\in Y$ such that $\{n\in E\mid x(n)\leq y(n)\}$ is infinite. 
\end{defn}
See \cite[Lemma 4.4]{ADLCM} for a proof of the following.
\begin{lem} \label{unbdd_lem}
    Suppose that $\Phi$ is a nontrivial coherent family indexed by some set $X$ and $Y\subseteq X$ is everywhere unbounded.
    Then $\Phi\upharpoonright Y$ is nontrivial. 
\end{lem}
We are now ready to prove Theorem \ref{Knaster_thm}:

\begin{proof}
    By reindexing along a closed cofinal subset of $\kappa$ consisting of limit ordinals using that a finite support iteration of Knaster forcings is Knaster, we may assume $\kappa=\operatorname{cof}(\kappa)$ and each $\dot\Qbb_i$ is forced to have an infinite antichain. 
    Let $\langle \dot q^i_n\mid n<\omega\rangle$ be $\Pbb_i$-names for an infinite antichain in $\Qbb_i$ and let $y_i$ be the real such that $y_i(n)=m$ if and only if 
    \[q^{\omega i+n}_m\in G_{\omega i+n}.\]
    Observe that $y_i$ is Cohen generic over $V[G\upharpoonright i]$. 
    In particular, since every $E\subseteq\omega$ and $f\colon E\to\omega$ appear at some proper initial stage of the iteration, $\{y_i\mid i<\kappa\}$ is everywhere unbounded. 

    Suppose for contradiction that in $V^{\mathbb{P}_\kappa}$, $\Phi$ is a nontrivial coherent family. 
    Then by Lemma \ref{unbdd_lem}, $\Phi\upharpoonright\{y_i\mid i<\kappa\}$ is nontrivial. 
    By an easy chain condition argument, there is a club $C\subseteq\kappa$ such that for all $\alpha\in C$, $\omega \alpha=\alpha$ and $\Phi^\alpha:=\Phi\upharpoonright\{y_i\mid i<\alpha\}\in V[G\upharpoonright \alpha]$. 
    \begin{claim}
        If $\alpha\in C$ has cofinality $\aleph_1$, then $\Phi^\alpha$ is trivial in $V[G\upharpoonright \alpha]$
    \end{claim}
    \begin{proof}
        Suppose not. 
        Let $\langle \beta_j\mid j<\omega_1\rangle$ be increasing and cofinal in $\alpha$. 
        By genericity, if $\Acal$ is an uncountable family of pairwise disjoint finite subsets of $\omega_1$, then 
        \[\left\{\bigwedge_{\beta\in F}y_{\beta_j}\right\}\]
        is everywhere unbounded in $V[G\upharpoonright i]$ so $\Phi^\Acal$ is nontrivial by Lemma \ref{unbdd_lem}. 
        But then in $V^{\Pbb_\kappa}$, since a finite support iteration of Knaster forcings is Knaster, Lemma \ref{no_ext_lem} implies that $\Phi^\alpha$ does not extend to a nontrivial coherent family indexed by $\{\beta_j\mid j<\omega_1\}\cup\{y_\alpha\}$, contradicting that $\Phi$ was a coherent family indexed by all of $\omega^\omega$.
    \end{proof}
    
    By the countable chain condition of the iteration and Fodor's lemma in $V^{\Pbb_\kappa}$, there is a stationary set $S\subseteq C\cap\operatorname{cof}(\omega_1)$ and a $\beta<\kappa$ such that for all $\alpha\in S$, $\Phi^\alpha$ is trivialized by an element of $V[G\upharpoonright\beta]$. 
    For each $\gamma\in S$, let $\Psi^\gamma$ be a trivialization of $\Psi$ in $V[G\upharpoonright\gamma]$. 
    The next claim completes the proof by showing that any $\Psi^\alpha$ trivializes $\Phi\upharpoonright\{y_i\mid i<\kappa\}$. 
    \begin{claim}
        If $\beta<\gamma<\delta$, there is an $n<\omega$ such that $\Psi^\gamma=^n\Psi^\delta$.
    \end{claim}
    \begin{proof}
        Suppose not. 
        Then there are infinitely many $n<\omega$ such that for some $m<\omega$, $\Psi^\gamma(n,m)\neq\Psi^\delta(n,m)$. 
        By genericity of $y_\beta$ over $V[G\upharpoonright\beta]$, there are infinitely many $n<\omega$ such that for some $m\leq y_\beta(m)$, $\Psi^\gamma(n,m)\neq\Psi^\delta(n,m)$. 
        In particular, $\Psi^\gamma$ and $\Psi^\delta$ cannot both trivialize a coherent family with an index containing $y_\beta$. 
        But they do both trivialize $\Phi^\gamma$, a contradiction. 
    \end{proof}
    
\end{proof}


\begin{thebibliography}{10}

\bibitem{ADLCM} N.~Bannister. {\em Additivity of derived limits in the Cohen model}, arXiv preprint: 2302.07222. To appear in Isr. J. Math.

\bibitem{strong_hom_add} N.~Bannister, J.~Bergfalk, J.~Tatch Moore. {\em On the additivity of strong homology
for locally compact separable metric spaces}, arXiv:2008.13089, February 22, 2021, 13pp.  

\bibitem{Baumgartner:PFA}
J.~E. Baumgartner.
{\em Applications of the proper forcing axiom},
Handbook of set-theoretic topology, 913--959.
North-Holland Publishing Co., Amsterdam, 1984.

\bibitem{BFZ}
J.E.~Baumgartner, R.~Frankiewicz, P.~Zbierski.
{\em Embedding of Boolean algebras in $\mathscr{P}(\omega)/\Fin$.}
Fund. Math. 136 (1990), no. 3, 187–192.

\bibitem{B17} J.~Bergfalk. {\em Strong homology, derived limits, and set theory}, Fund. Math. 236 (2017), no. 1, 17--28.

\bibitem{SVHDL} J.~Bergfalk, C.~Lambie-Hanson. {\em Simultaneously vanishing higher derived limits}, 
Forum Math, Pi 9 (2021), Paper no. e4, 31pp.

\bibitem{SVHDLwLC} J.~Bergfalk, M.~Hru\v{s}\'{a}k, C.~Lambie-Hanson. {\em Simultaneously vanishing higher derived limits without large cardinals}, arXiv:2102.06699, February 12, 2021, 30 pp.

\bibitem{BK} A.~K.~Bousfield, D.~M.~Kan. {\em Homotopy limits, completions and localizations}. Lecture Notes in Mathematics, Vol. 304. Springer-Verlag, Berlin-New York, 1972. v+348 pp.

\bibitem{CaEi} H.~Cartan, S.~Eilenberg, {\em Homological algebra}. With an appendix by David A. Buchsbaum. Princeton University Press, Princeton, NJ, 1956. xvi+390 pp.

\bibitem{SNVHDL} M.~Casarosa and C.~Lambie-Hanson. {\em Simultaneously nonvanishing higher derived limits}. 2024, arXiv preprint:2411.15856.

\bibitem{AnStacks} D.~Clausen and P.~Scholze. {\em Analytic Stacks.} \url{https://www.youtube.com/watch?v=EW39K0J7Hqo&list=PLx5f8IelFRgGmu6gmL-Kf_Rl_6Mm7juZO&index=4} 2024. Posted by 
Institut des Hautes Etudes Scientifiques (IHES). Accessed: 23 April 2026.

\bibitem{Connon} E.~Connon. {\em On $d$-dimensional cycles and the vanishing of simplicial homology}, arxiv:1211.7087, July 22, 2013, 19 pp.

\bibitem{DSV} A.~Dow, P.~Simon, J.~Vaughan. {\em Strong homology and the proper forcing
axiom}, Proc. Amer. Math. Soc. 106.3 (1989), 821–-828.


\bibitem{EiMo2} S.~Eilenberg, J.~C.~Moore. {\em Foundations of relative homological algebra}.  Mem. Amer. Math. Soc. 55 (1965), 39 pp.

\bibitem{EiMo} \bysame. {\em Limits and spectral sequences}. Topology 1 (1962), 1--23.

\bibitem{Gob} R\'{e}mi Goblot. {\em Sur les d\'{e}riv\'{e}s de certaines limites projectives. Applications aux modules}, Bull. Sci. Math. (2) 94 (1970), 251--255.

\bibitem{Hechler} S.~Hechler. {\em On the existence of certain cofinal subsets of $\omega^\omega$}. Axiomatic Set Theory, Proceedings of
Symposia in Pure Mathematics, Vol. 13, Part II (Amer. Mathematical Soc., Providence, RI, 1974) 155-
173. 

\bibitem{GJ} P.~Goerss, J.~Jardine. {\em Simplicial Homotopy Theory}. Birkh\"{a}user Verlag, Basel, 1999. xv+510 pp.

\bibitem{higher_infinite} A.~Kanamori. {\em The higher infinite. Large cardinals in set theory from their beginnings.}
Perspectives in Mathematical Logic. Springer-Verlag, Berlin, 1994. xxiv+536pp.

\bibitem{kechris} A.~Kechris. {\em Classical descriptive set theory.} Graduate Texts in Mathematics, 156. Springer-Verlag, New York, 1995. xviii+402 pp.

\bibitem{set_theory:kunen} K.~Kunen. {\em Set theory. An introduction to independence proofs.} Studies in Logic and the Foundations of Mathematics, 102. North-Holland Publishing Co., Amsterdam-New York, 1980. xvi+313 pp.

\bibitem{Laver}
R.~Laver.
{\em Linear orders in 
$(\omega)^\omega$ under eventual 
dominance.}
Logic Colloquium '78 (Mons, 1978), pp. 299--302
Stud. Logic Found. Math., 97
North-Holland Publishing Co., Amsterdam-New York, 1979.

\bibitem{Ker} J.~Lurie. {\em Kerodon}. Online 2018 text: \url{https://kerodon.net}, accessed Jan 2023.

\bibitem{SHINA} S.~Marde\v{s}i\'{c} and A.~V.~Prasolov. {\em Strong homology is not additive}, Trans. Amer. Math. Soc. 307.2 (1988), 725--744.

\bibitem{MaPo} J.~P.~May, K.~Ponto. {\em More concise algebraic topology. Localization, completion, and model categories}. Chicago Lectures in Mathematics. University of Chicago Press, Chicago, IL, 2012. xxviii+514 pp.

\bibitem{McG} C.~A.~McGibbon. {\em Phantom maps}, Handbook of algebraic topology, 1209–-1257, North-Holland, Amsterdam, 1995.

\bibitem{Mil} J.~Milnor. {\em On axiomatic homology theory}.  Pacific J. Math. 12 (1962), 337--341.

\bibitem{Mitch} B.~Mitchell. {\em Rings with several objects}, Advances in Math. 8 (1972), 1--161.

\bibitem{Nob}  G.~N\"{o}beling. \emph{\"{U}ber die Derivierten des Inversen und des direkten Limes einer Modulfamilie}, Topology 1 (1962), 47--61.

\bibitem{Oso} B.~L.~Osofsky. {\em The subscript of $\aleph_n$, projective dimension, and the vanishing of ${\lim}^n$}, Bull. Amer. Math. Soc. 80 (1974), 8--26.

\bibitem{Qui} D.~Quillen. {\em Homotopical algebra.} Lecture Notes in Mathematics 43. Springer, Berlin, 1967.

\bibitem{Roos} J.~E.~Roos. {\em Sur les foncteurs d\'{e}riv\'{e}s des lim.} C.~R.~Acad. Sci. Paris 252 (1961), 3702--3704.

\bibitem{Kamo} S.~Kamo, Almost coinciding families and gaps in $P(\omega)$, J. Math. Soc. Japan 45 (1993),
no. 2, 357–368.

\bibitem{SSH} S.~Marde\v{s}i\'{c}. {\em Strong shape and homology.}
Springer Monographs in Mathematics. Springer-Verlag, Berlin, 2000. xii+489 pp.

\bibitem{Pra} A.~V.~Prasolov. {\em Non-additivity of strong homology}, Topology Appl. 153 (2005), 493--527.

\bibitem{Ste} N.~E.~Steenrod. {\em Regular cycles of compact metric spaces}. Ann. of Math. 41 (1940), 833--851.

\bibitem{Scheepers}
M.~Scheepers. {\em Gaps in $\omega^\omega$}, Set theory of the reals (Ramat Gan, 1991).
Israel Math. Conf. Proc., 6,  439--561.

\bibitem{To98} S.~Todorcevic. {\em The first derived limit and compactly $F_\sigma$ sets}, J. Math. Soc. Japan 50 (1998), no. 4, 831–836.

\bibitem{NVHDL}
B.~Veli\v{c}kovi\'c, A.~Vignati. 
{\em Non-vanishing higher derived limits.}
\newblock Communications in Contemporary Mathematics 16.7 (2023), World Scientific Publishing Company.

\bibitem{Wei}  C.~A.~Weibel. {\em An introduction to homological algebra}. Cambridge Studies in Advanced Mathematics, 38. Cambridge University Press, Cambridge, 1994. xiv+450 pp.

\bibitem{Yeh}  Z.~Z.~Yeh. {\em Higher inverse limits and homology theories.} Thesis (Ph.D.), Princeton University, 1959. 73 pp.

\bibitem{forcing_idealized} J.~Zapletal. {\em Forcing Idealized.} Cambridge Tracts in Mathematics, 174. Cambridge University Press, Cambridge, 2008. vi+314 pp.

\end{thebibliography}
\end{document}